\documentclass{amsart}
\usepackage{graphicx, amsmath, amssymb, amsthm, amsfonts,fullpage, enumerate, xcolor, tikz, tikz-cd, verbatim, hyperref,standalone, mathtools}
\usetikzlibrary{calc}
\usepackage[dvipsnames]{xcolor}
\usepackage{thmtools}
\usepackage{cleveref}
\usepackage[backend=biber, maxbibnames=99]{biblatex}
\theoremstyle{definition}
\newtheorem{defn}{Definition}[section]

\newtheorem{example}[defn]{Example}
\theoremstyle{plain}
\newtheorem{lemma}[defn]{Lemma}

\newtheorem{proposition}[defn]{Proposition}
\newtheorem{corollary}[defn]{Corollary}

\newtheorem{thrm}[defn]{Theorem}
\newtheorem{theorem}[defn]{Theorem}
\newtheorem{conjecture}[defn]{Conjecture}

\DeclareMathOperator\NC{\mathrm{NC}}
\DeclareMathOperator\PF{\mathrm{PF}}
\DeclareMathOperator\R{\mathrm{R}}

\DeclareMathOperator\rk{\mathrm{rk}}

\DeclareMathOperator\Des{\mathrm{Des}}
\DeclareMathOperator\des{\mathrm{des}}
\DeclareMathOperator\asc{\mathrm{asc}}
\DeclareMathOperator\wdes{\mathrm{wdes}}
\DeclareMathOperator\wasc{\mathrm{wasc}}
\DeclareMathOperator\tie{\mathrm{tie}}

\DeclareMathOperator\NDWD{\widetilde{\mathrm{PF}}}
\DeclareMathOperator\NT{\mathrm{NT}}

\begin{document}

\title{The Chow Polynomial of the Noncrossing Partition Lattice}
\author{Joseph Chun, Natsuka Hayashida, Ethan Partida, Zhixing Wang}
\date{\today}
\begin{abstract}
    We prove that the Chow polynomial of the noncrossing partition lattice is equal to the descent generating function of tieless parking functions. As a consequence, we obtain new results on the descent statistic for parking functions. 
\end{abstract}

\maketitle

\vspace{-2em}

\section{Introduction}
Let $P$ be a finite, graded and bounded poset. The \emph{Chow polynomial} of $P$ is the polynomial
\begin{equation*}H_P(x) = \sum_{\hat 0 = p_0 <p_1 <\cdots < p_m \leq \hat 1} \prod_{i=1}^m x[\rk(p_i)-\rk(p_{i-1})-1]_x \in \mathbb{R}[x]\end{equation*}
where the sum ranges over all chains of $P$ that start with the minimum element of $P$ and for an integer $m\geq 0$, $[m]_x\coloneqq 1+x+x^2+\ldots+x^{m-1}$. The Chow polynomial has its roots in matroid theory. If $\mathcal{L}$ is the lattice of flats of a matroid $M$, then $H_\mathcal{L}(x)$ is the Hilbert--Poincar\'e series of the Chow ring $\mathrm{CH}(M)$ of $M$. The Chow ring of a matroid is an important and highly structured ring. In their landmark work, Adiprasito--Huh--Katz \cite{AHK} proved that $\mathrm{CH}(M)$ possesses the K\"ahler package, and thus, $H_{\mathcal{L}}(x)$ is a palindromic and unimodal polynomial. These properties of $H_{\mathcal{L}}(x)$ were given combinatorial proofs by Ferroni--Matherne--Vecchi in \cite{FMV}. Later, Ferroni--Matherne--Stevens--Vecchi \cite{FMSV} showed that $H_P(x)$ is also palindromic and unimodal for all finite, graded, and bounded posets $P$.

The goal of this paper is to study the Chow polynomial of the noncrossing partition lattice $\NC_{n+1}$. The noncrossing partition lattice is an interesting and well-behaved subposet of the lattice of set partitions, with an extensive literature (see, e.g.,  the surveys \cite{ncSurprising, simion} and the references within). However, the noncrossing partition lattice is not the lattice of flats of a matroid when $n\geq 3$; it fails to be semi-modular.  

In a seminal article, Stanley gave a labeling of the Hasse diagram of $\NC_{n+1}$ such that maximal chains correspond to length $n$ parking functions \cite{Stanley}. A parking function $\rho\in \PF_n$ of length $n$ is a tuple of integers $\rho=(\rho_1,\ldots, \rho_n)$ such that the sequence $\tilde \rho$ obtained by sorting $\rho$ into non-decreasing order has the property that $\tilde\rho_i \leq i$ for all $i\in [n]$. We say that $\rho$ has a descent in position $i$ if $\rho_i>\rho_{i+1}$, and let $\des(\rho)$ denote the number of descents of $\rho$. A parking function $\rho$ is called tieless if $\rho_i\neq \rho_{i+1}$ for all $i\in[n-1]$. Let $\NT_n$ be the subset of $\PF_n$ of length $n$ tieless parking functions. Our main result provides a new connection between the noncrossing partition lattice and parking functions.

\begin{theorem}\label{thrm:main1}
    The Chow polynomial of the noncrossing partition lattice is equal to the descent generating function of tieless parking functions
    \[H_{\NC_{n+1}}(x) = \sum_{\rho\in NT_n} x^{\des(\rho)}.  \]
\end{theorem}

It is interesting to compare our knowledge of the Chow polynomial of the noncrossing partition lattice with that of the closely related partition lattice. Despite ample study (see \cite{DHMPR, FMSV, GS12, KK, LiaoGamma, Stump}), there is no combinatorial interpretation of the coefficients of the Chow polynomial of the partition lattice analogous to Theorem~\ref{thrm:main1}. In this regard, it is notable that Theorem~\ref{thrm:main1} gives an interpretation for the Chow polynomial of $\NC_{n+1}$ as a combinatorially defined generating function. See Figure~\ref{fig:chow_polys} for a list of the Chow polynomials of $\NC_{n+1}$ for $n\leq 9$.

    \begin{figure}
    \begin{tabular}{lc}
       $n=1$&$ 1$\\\\
       $n=2$&$  {\color{blue}\underline{1}} + {\color{blue}\underline x}$\\\\
       $n=3$&$ x^2 + {\color{blue}\underline{7x}} + 1$\\\\ 
$n=4$&$ x^3 + {\color{blue}\underline{31x}^2} + {\color{blue}\underline{31x}} + 1$\\\\  
$n=5$&$ x^4 + {\color{blue}\underline{116x}^3} + 391x^2 + {\color{blue}\underline{116x}} + 1$\\\\
$n=6$&$ x^5 + {\color{blue}\underline{407x}^4} + 3480x^3 + 3480x^2 + {\color{blue}\underline{407x}} + 1$\\\\  
$n=7$&$ x^6 + {\color{blue}\underline{1401x}^5} + 26097x^4 + 62651x^3 + 26097x^2 + {\color{blue}\underline{ 1401x}} + 1$\\\\  
$n=8$&$ x^7 + {\color{blue}\underline{4825x}^6} + 178621x^5 + 865129x^4 + 865129x^3 + 178621x^2 + {\color{blue}\underline{4825x}} + 1$\\\\
$n=9$&$ x^8 + {\color{blue}\underline{16750x}^7} + 1162684x^6 + 10228978x^5 + 20229895x^4 + 10228978x^3 + 1162684x^2 + {\color{blue}\underline{16750x}} + 1$
\end{tabular}

\caption{The Chow polynomials of $\NC_{n+1}$ for $n\leq 9$. The nonzero terms counted by Corollary~\ref{thrm:linearCoeff} are underlined and drawn in blue.}
\label{fig:chow_polys}
\end{figure}

The descent statistic for parking functions is a natural generalization of the descent statistic for permutations, the study of which dates back to MacMahon \cite{MacMahon}. Descents in parking functions are studied in the works of Schumacher \cite{Schumacher} and Cruz--Harris--Kretschmann--McClinton--Moon--Museus--Redmon \cite{CHHKMMM}. Of particular interest is the parking function distribution $T_n(i,j)$, which counts the number of parking functions of length $n$ with $i$ ties and $j$ descents. Figure \ref{fig:distribution}, reproduced from \cite[Figure 2]{Schumacher} and \cite[Figure 2]{CHHKMMM}, lists the non-zero values of $T_6(i,j)$ arranged such that $i$ increases from bottom to the top and $j$ increases from the left edge of a row to the right.

\begin{figure}
\begin{tabular}{ccccccccccccccc}
	&&&&&&&& \rotatebox{45}{$j=0$} &&&&&& \\
	{i=5} &&&&&&& 1 && \rotatebox{45}{$j=1$}& \\
	{i=4} &&&&&& 15 && 15 && \rotatebox{45}{$j=2$}& \\
	{i=3} &&&&& 50 && 260 && 50 && \rotatebox{45}{$j=3$}& \\
	{i=2} &&&& 50 && 1030 && 1030 && 50 && \rotatebox{45}{$j=4$}& \\
	{i=1} &&& 15 && 1240 && 3970 && 1240 && 15 && \rotatebox{45}{$j=5$}& \\
	{i=0} && \underline{\color{blue}1} && \underline{\color{blue}407} && \underline{\color{blue}3480} && \underline{\color{blue}3480} && \underline{\color{blue}407} && \underline{\color{blue}1}& \color{white}\rotatebox{45}{$j=6$}\\
\end{tabular}
    \caption{Parking function distribution for $\PF_6$. The bottom row, which is underlined and drawn in blue, forms the coefficients of $H_{\NC_{7}}(x)$} 
    \label{fig:distribution}
\end{figure}

 For general $n$, we refer to the triangular array analogous to Figure~\ref{fig:distribution} as the array of $T_n(i,j)$. The numbers along the diagonal edges of the array of $T_n(i,j)$ are the Narayana numbers \cite[Theorem 12]{Schumacher}. It is an outstanding open problem to give closed formulas for other entries of the array of $T_n(i,j)$ \cite[Problem 1]{CHHKMMM}. In Corollary~\ref{thrm:linearCoeff}, we make progress on this problem by giving a Catalan-esque formula for the numbers $T_n(0,1)=T_n(0,n-2)$. To the best of our knowledge, Corollary~\ref{thrm:linearCoeff} is the first closed formula for an entry of the array of $T_n(i,j)$ outside of its diagonal edges.

 \begin{corollary}\label{thrm:linearCoeff}
  The number of tieless parking functions $\rho\in \PF_n$ with $1$ descent (equivalently, $n-2$ descents) is
  \[C_{n+1}-\binom{n+1}{2}-1  \]
  where $C_{n+1}= \frac{1}{n+2}\binom{2n+2}{n+1}$ is the $(n+1)$th Catalan number.
\end{corollary}
 
For example, we can confirm using Figure~\ref{fig:distribution} that the number of tieless parking functions $\rho\in \PF_6$ with $1$ descent (equivalently, $4$ descents) is 
\[C_7-\binom{7}{2}-1= 429 - 21 - 1 = 407 = T_6(0,1) = T_6(0,4) . \]

The Narayana polynomial, whose coefficients form either of the diagonal edges of the array of $T_n(i,j)$, is known to be real-rooted \cite[Section 4]{B02}, \cite[Theorem 5.3.1]{Brenti89}. By Theorem~\ref{thrm:main1}, the coefficients of the Chow polynomial $H_{\NC_{n+1}}(x)$ form the bottom edge of the array of $T_n(i,j)$. We conjecture that $H_{\NC_{n+1}}(x)$ is real-rooted.

\begin{conjecture}
\label{conj:real-rooted}
The descent generating function of tieless parking functions is real-rooted.
\end{conjecture}

By explicitly computing the Chow polynomial of $\NC_{n+1}$, we have verified Conjecture~\ref{conj:real-rooted} up to $n=11$.

\begin{proposition}
    For $n\leq 11$, $H_{\NC_{n+1}}(x)$ is real-rooted.
\end{proposition}

Again, it is useful to draw a comparison to what is known about the Chow polynomial of the partition lattice. In recent work, Coron--Ferroni--Li \cite{CFL25} prove that the Chow polynomials of UMEL-shellable posets are real-rooted. The partition lattice is UMEL-shellable, and thus its Chow polynomial is real-rooted \cite[Theorem 1.8]{CFL25}. However, when $n\geq 3$, the noncrossing partition lattice $\NC_{n+1}$ is not UMEL-shellable and the results of \cite{CFL25} do not apply to Conjecture~\ref{conj:real-rooted}.

Real-rootedness of a polynomial is one of the strongest ``positivity'' properties for polynomials. A weaker property of interest is that of $\gamma$-positivity. In our second main result, we prove that the descent generating function of tieless parking functions is $\gamma$-positive. As a consequence of $\gamma$-positivity, we conclude that the number of tieless parking functions, counted by their descents, forms a unimodal sequence. Let $\NDWD_n$ be the subset of parking functions of length $n$ without two consecutive weak descents and without a final weak descent, i.e., there is no $i\in [n]$ where $\rho_{i-1}\geq \rho_i \geq \rho_{i+1}$, assuming that $\rho_0=0=\rho_{n+1}$.

\begin{theorem}\label{thrm:main2}
  The descent generating function of tieless parking functions has the following expansion:
    \begin{equation}\label{eq:gammaIntro}
    \sum_{\rho\in \NT_n} x^{\des{\rho}} = \sum_{\rho \in \NDWD_n} x^{\wdes(\rho)}(x+1)^{n-2\wdes(\rho)-1}. 
    \end{equation}
  In particular, it is $\gamma$-positive and its coefficients form a unimodal sequence.
\end{theorem}

Our proof of Theorem~\ref{thrm:main2} relies on a generalization of the well known ``valley-hopping'' method; see the surveys \cite[Section 3]{Branden15} and \cite[Section 4]{Athanasiadis18}. Our generalization is reminiscent of, but not the same as, the generalization of the valley-hopping method to Smirnov words appearing in \cite[Section 3]{Athanasiadis16}. Theorem~\ref{thrm:main2} is a key input to our proof of Theorem~\ref{thrm:main1}. After applying a theorem of Stump \cite[Theorem 1.1]{Stump} to Stanley's R-labeling of $\NC_{n+1}$ indexed by parking functions \cite{Stanley}, we find that $H_{\NC_{n+1}}(x)$ is equal to the right hand side of \Cref{eq:gammaIntro} (Proposition~\ref{prop:PFChow}). From Proposition~\ref{prop:PFChow} and Theorem~\ref{thrm:main2}, Theorem~\ref{thrm:main1} follows immediately. 

This paper is organized as follows. In Section~\ref{sec:background}, we recall necessary background. In Section~\ref{sec:gamma}, we give a positive $\gamma$-expansion of the descent generating function of tieless parking functions and prove Theorem~\ref{thrm:main2}. In Section~\ref{sec:pf}, we recall Stanley's R-labeling of the noncrossing partition lattice and prove Theorem~\ref{thrm:main1}.

\subsection*{Acknowledgments}

Much of this work was completed during the 2025 Summer SPRINT program at Brown University under the guidance of Melody Chan. 
The authors are grateful to Melody Chan, Luis Ferroni, Yifan Guo, Pamela E. Harris, Trevor Karn, Carly Klivans, and Jacob P. Matherne for helpful conversations. The inception of this work was a talk given by Luis Ferroni titled ``The Chow polynomial of a poset'' at the 2024 Santander Workshop on Geometric and Algebraic Combinatorics. We thank the conference organizers for creating a welcoming and productive environment. Joseph Chun and Zhixing Wang were supported by the 2025 Summer SPRINT/UTRA at Brown University. Natsuka Hayashida was supported by NSF DMS-2401282. Ethan Partida was partially supported by NSF DMS-2053288, a U.S. Department of Education GAANN award, and the Simons Foundation SFI-MPS-SDF-00015018. 

\section{Background}\label{sec:background}
\subsection{Properties of polynomials}
Given a polynomial $f(x)=\sum_{i=0}^n a_i x^i \in \mathbb{R}[x]$ of degree $n$, we say that it is
\begin{itemize}
    \item \emph{palindromic} if $a_{i}=a_{n-i}$ for all $0\leq i \leq \lfloor n/2 \rfloor$,
    \item \emph{unimodal} if there exists an index $j$ such that
    \[a_0 \leq a_1 \leq \ldots \leq a_{j-1} \leq a_j \geq a_{j+1} \geq \ldots \geq a_n, \]
    \item \emph{$\gamma$-positive} if it is palindromic and can be written as \[f(x) = \sum_{i=0}^{\lfloor n/2 \rfloor} \gamma_i x^i(x+1)^{n-2i}\] where each $\gamma_i$ is a non-negative real number,
    \item \emph{real-rooted} if all of the zeroes of $f(x)$ are real.
\end{itemize}
For a palindromic polynomial $f(x)$, there is a chain of implications:
\[\text{real-rooted } \implies \text{$\gamma$-positive } \implies \text{ unimodal.} \]
We refer the reader to the survey \cite{Branden15} for more on the interplay between these properties.

\subsection{The noncrossing partition lattice}
In this section, we review the definition of and relevant facts about the noncrossing partition lattice. We refer the readers to the surveys of McCammond \cite{ncSurprising} and Simion \cite{simion} for more details.

The \emph{partition lattice} $\Pi_{n+1}$ consists of all partitions of the set $[n+1]=\{1,2,\ldots,n+1\}$ into blocks, ordered by refinement. Let $v_1,\ldots, v_{n+1}$ be the vertices of a regular $n+1$-gon inscribed in the unit circle and labeled in a clockwise fashion. We say that a partition $\pi$ is \emph{noncrossing} if for any two blocks $A$ and $B$ of $\pi$,
the convex hull of the points $\{v_A: a \in A\}$ does not intersect the convex hull of $\{v_b: b\in B\}$. We refer the reader to Figure~\ref{fig:ncPart} for a visualization where we have identified $\{v_1,\ldots, v_{n+1}\}$ with the set $\{1,2,\ldots, n+1\}$.  

\begin{figure}
    \centering
    \begin{tikzpicture} [scale=1.2]
    \filldraw[very thick, draw opacity = 0.8, fill opacity =0.4, color=blue!40!gray] (-1*1/9*360:1cm)--(-2*1/9*360:1cm)--(-6*1/9*360:1cm)--cycle;
    \filldraw[very thick, draw opacity = 0.8, fill opacity =0.4, color=blue!40!gray] (-3*1/9*360:1cm)--(-5*1/9*360:1cm)--cycle;
    \filldraw[very thick, draw opacity = 0.8, fill opacity =0.4, color=blue!40!gray] (-7*1/9*360:1cm)--(-8*1/9*360:1cm)--(-9*1/9*360:1cm)--cycle;
    \draw[thick] circle (1cm);
        \foreach \x in {1,2,...,9} {
        \draw node at (-\x*1/9*360:1.25cm) {$\x$};
        \fill (-\x*1/9*360:1cm) circle (2pt);
    };
    \end{tikzpicture}
    \caption{The noncrossing partition $126\vert 35 \vert 4 \vert 789$}
    \label{fig:ncPart}
\end{figure}
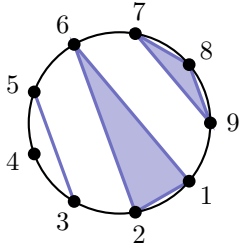

The subposet $\NC_{n+1}$ of $\Pi_{n+1}$ consisting of noncrossing partitions is the \emph{noncrossing partition lattice}. The noncrossing partition lattice is a graded, bounded lattice. The rank of a partition $\pi \in \NC_{n+1}$ is equal to $n+1-\vert\pi\vert$ where $\vert\pi\vert$ is the number of blocks of $\pi$. The noncrossing partition lattice possesses many striking enumerative properties. Among these is the following.

\begin{theorem}(\cite[Corollary 4.2]{Kreweras})\label{thrm:catalan}
The size of the noncrossing partition lattice $\NC_{n+1}$ is equal to the $(n+1)$st Catalan number
\[\#\NC_{n+1} = C_{n+1} = \frac{1}{n+2}\binom{2n+2}{n+1}. \]
\end{theorem}

\subsection{Edge labelings}
Let $P$ be a finite, graded, and bounded poset of rank $n$, $\mathcal{E}(P)=\{(x,y): y \text{ covers } x \text{ in } P\}$ be the set of edges of the Hasse diagram of $P$, and $\lambda:\mathcal{E}(P)\to \mathbb{Z}$ be a map, which we refer to as a \emph{labeling} of $P$. If $\mathbf{m}= (x_0<x_1<\ldots < x_k)$ is a saturated chain in $P$, we write $\lambda(\mathbf{m}) = (\lambda(x_0,x_1),\lambda(x_1,x_2),\ldots, \lambda(x_{k-1},x_k))$. We say that $\mathbf{m}$ is \emph{weakly increasing} if $\lambda(x_0,x_1)\leq \lambda(x_1,x_2)\leq \cdots\leq  \lambda(x_{k-1},x_k)$.

\begin{defn}
    A map $\lambda: \mathcal{E}(P)\to \mathbb{Z}$ is an \emph{R-labeling} if every interval of $P$ contains exactly one weakly increasing maximal chain $\mathbf{m}$.
\end{defn}

For a maximal chain $\mathbf{m}$ of $P$ and labeling $\lambda$ of $P$, let $\Des(\mathbf{m}) = \{i \in [n-1]: \lambda(\mathbf{m})_i > \lambda(\mathbf{m})_{i+1}\}$ and $\des(\mathbf{m}) = \# \Des(\mathbf{m})$. We are interested in R-labelings because of the following result of Stump.

\begin{theorem}{\cite[Theorem 1.1]{Stump}}\label{thrm:stump}
    Let $P$ be a finite, graded, and bounded poset with $R$-labeling $\lambda$. The Chow polynomial $H_P(x)$ has the expansion
    \[H_P(x) = \sum_{\mathbf{m}} x^{\des(\mathbf{m})}(x+1)^{n-1-2\des(\mathbf{m})}\]
    where the sum is over all maximal chains $\mathbf{m}$ in $P$ such that $1\not\in \Des(\mathbf{m})$, and there is no index $i$ such that $\{i,i+1\}\subset \Des(\mathbf{m})$. 
\end{theorem}

We remark that Theorem~\ref{thrm:stump} is related to, and generalized by, a formula of Ferroni--Matherne--Vecchi for the (not necessarily positive) $\gamma$-expansion of the Chow polynomial $H_P(x)$ of an arbitrary finite, graded and bounded poset $P$ \cite[Theorem 4.25]{FMV}. 

\subsection{Parking Functions}
\begin{defn}
  A \emph{parking function of length $n$} is a function $\rho:[n]\to [n]$ such that the preimage $\rho^{-1}([k])$ has cardinality at least $k$ for all $1\leq k \leq n$. We let $\PF_n$ denote the set of all length $n$ parking functions.
\end{defn}
As done in the introduction, we often identify a parking function $\rho:[n]\to[n]$ with the tuple of integers $(\rho_1\coloneq \rho(1),\rho_2\coloneq \rho(2),\ldots, \rho_n\coloneq \rho(n))$. There is extensive literature studying parking functions, and we refer the reader to the survey by Mart\'inez Mori \cite{MM} for references and more details. 

For $\rho \in \PF_n$ and $i \in [n-1]$, we say that
\begin{itemize}
    \item $\rho$ has an \emph{ascent in position $i$} if $\rho(i) < \rho(i+1)$,
    \item $\rho$ has a \emph{descent in position $i$} if $\rho(i) > \rho(i+1)$,
    \item $\rho$ has a \emph{tie in position $i$} if $\rho(i) = \rho(i+1)$, and 
    \item $\rho$ has a \emph{weak descent in position $i$} if $\rho(i) \geq \rho(i+1)$.
    \end{itemize}
    We let $\asc(\rho)$, $\des(\rho)$, $\tie(\rho)$, and $
    \wdes(\rho)$ be the number of ascents, descents, ties, and weak descents of $\rho$, respectively.  Note that \[\wdes(\rho) = \des(\rho)+\tie(\rho) = n-1 - \asc(\rho).\] 
    Two subsets of parking functions play a prominent role in this paper.

\begin{defn}
A parking function $\rho\in \PF_n$ is \emph{tieless} if $\tie(\rho)=0$. We denote the set of all length $n$ tieless parking functions by $\NT_n$.
\end{defn}

\begin{defn}
A parking function $\rho\in \PF_n$ \emph{has no consecutive weak descents} if $n-1$ is an ascent of $\rho$ and whenever $i\in[n-2]$ is not an ascent of $\rho$, $i+1$ is an ascent of $\rho$. We denote the set of all length $n$ parking functions that have no consecutive weak descents by $\widetilde{\PF}_n$.
\end{defn}

It will be convenient for us to extend a parking function $\rho\in \PF_n$ to a function $\rho:\{0,1,\ldots, n,n+1\} \to [n]$ by defining $\rho(0)=0=\rho(n+1)$. This convention often allows for more homogeneous definitions. For example, $\rho \in \widetilde{\PF}_n$ if and only if there is no $i\in [n]$ such that $\rho(i-1) \geq \rho(i) \geq \rho(i+1)$. Despite our extension of $\rho$ to a function whose domain is $\{0,1,\ldots, n,n+1\}$, we emphasize that $\rho$ does not have an ascent in position $0$ nor a descent in position $n$.

\section{A $\gamma$-expansion for descents in tieless parking functions}\label{sec:gamma}

The goal of this section is to give a positive $\gamma$-expansion of the descent generating function of the set of tieless parking functions and prove Theorem~\ref{thrm:main2}. Our argument is a generalization of the well-known ``valley-hopping'' method for counting descents in permutations to the case of tieless parking functions. We refer the reader to the surveys \cite[Section 3]{Branden15} and \cite[Section 4]{Athanasiadis18} for an exposition on this method. An important subtlety in our argument is that we do not count descents in the set of all parking functions, only those which are ``reachable'' from a parking function with no consecutive weak descents (Proposition~\ref{prop:countDes}). The set of reachable parking functions is a strict subset of the set of all parking functions and is a strict superset of the set of all tieless parking functions (Example~\ref{ex:mainTheorem}).

\begin{defn} \label{def:ai}
  Let $\rho\in \PF_n$, $i\in[n]$ and suppose that $\ell\geq 1$ is an integer such that $\rho(i-\ell)\neq \rho(i-\ell+1)=\rho(i-\ell+2)=\ldots=\rho(i)\neq \rho(i+1)$. We say that the position $i$ of $\rho$ is \emph{ascent-intermediary} if $\rho(i-\ell)< \rho(i) < \rho(i+1)$  and \emph{descent-intermediary} if $\rho(i-\ell)> \rho(i) > \rho(i+1)$.
\end{defn}

  Definition~\ref{def:ai} ensures that only the ends of consecutive ties can be ascent-intermediary or descent-intermediary.  It is instructive to consider the ``mountain range'' $M_\rho$ of $\rho$ obtained by connecting all the consecutive pairs of points $(0,0),(1,\rho(1)),\ldots,(n,\rho(n)),(n+1,0)$ via straight-line segments in $\mathbb{R}^2$. In this way, the ascent-intermediary positions of $\rho$ are those corresponding to nodes on the up-slopes of $M_\rho$, excluding those that start a tie. Similarly, the descent-intermediary positions of $\rho$ correspond to nodes on the down-slopes of $M_\rho$, excluding those that start a tie. We illustrate the definition of an ascent/descent-intermediary position in the following example.

\begin{example}
Consider the length $9$ parking function $\rho=(1,2,2,4,3,2,2,3,2)$. The mountain range $M_\rho$ of $\rho$ is pictured in Figure~\ref{fig:mountain}. The first three positions of $\rho$ correspond to up-slopes of $M_\rho$. However, position $2$ is the start of a tie so out of the first three positions, only positions $1$ and $3$ are ascent-intermediary. Also note that although $\rho(6) = \rho(7) < \rho(8)$, $7$ is not an ascent intermediary position of $\rho$. This is because $\rho(5)> \rho(6)=\rho(7)$. Visually, one can intuit this by noticing that $7$ is in a valley of $M_\rho$ and not an up-slope of $M_\rho$. In a similar fashion, we see that $5$ and $9$ are the descent-intermediary positions of $\rho$. 
\begin{figure}[h!]
\centering
\includegraphics[scale = 0.27]{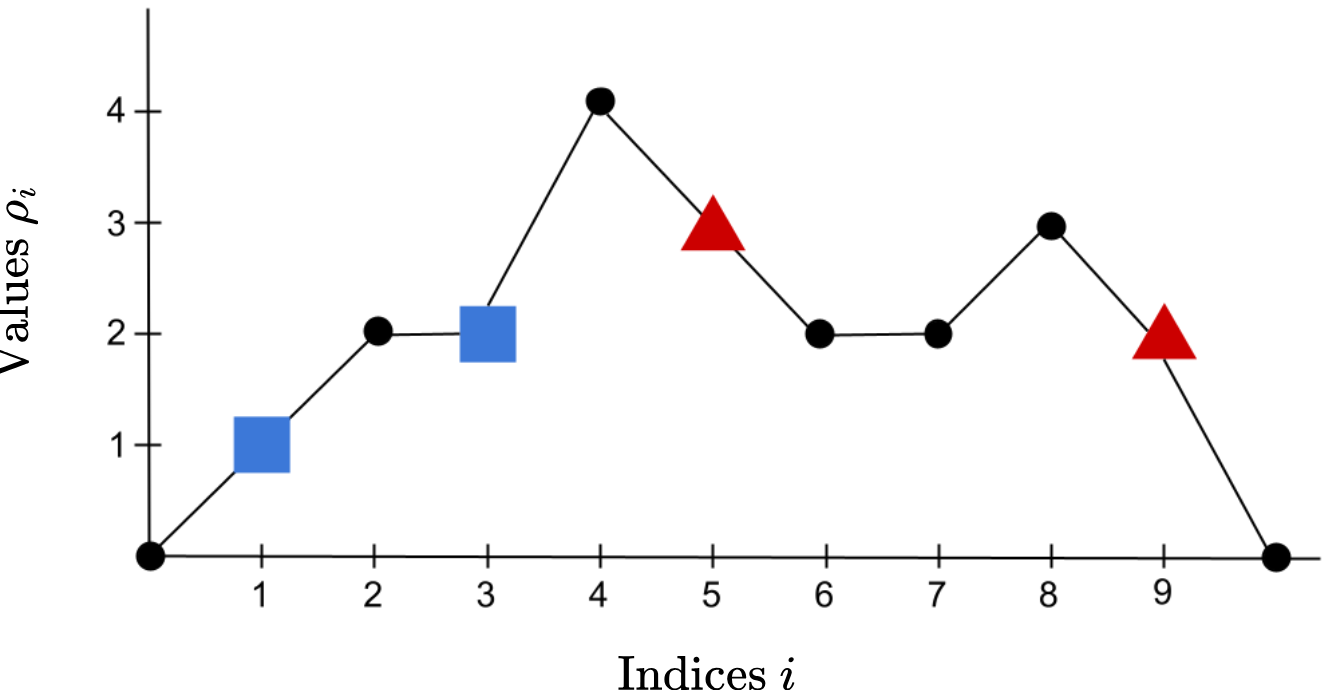}
\caption{The mountain range representation of the parking function $(1,2,2,4,3,2,2,3,2)$. The blue squares correspond to the ascent-intermediary positions, and the red triangles correspond to the descent-intermediary positions.} 
  \label{fig:mountain}
\end{figure}
\end{example}

Although quite technical, the following definition has an intuitive description in terms of mountain diagrams. We discuss this intuitive description in the paragraph immediately following the formal definition.
\begin{defn}\label{defn:lr}
  For a parking function $\rho$ and an ascent-intermediary position $i$ of $\rho$, let $i'$ be the smallest integer such that $i+1<i'\leq n+1$ and $\rho(i')<\rho(i)$. Define $R_i(\rho)\in \PF_n$ to be the parking function obtained by composing $\rho$ with the cyclic permutation $(i\ i+1 \ \ldots\  i'-1 )$: \[R_{i}(\rho)\coloneq\rho\circ (i\ i+1\ \ldots \ i'-1).\]
  For an integer $k\in [n]$ which is not an ascent-intermediary position of $\rho$, let $R_k(\rho)\coloneq\rho$. For an integer $i\in [n]$, let $R_i:\PF_n \to \PF_n$ be the map given by sending $\rho \mapsto R_i(\rho)$. For reasons we will shortly explain, we refer to the functions $R_1,R_2,\ldots, R_n$ as \emph{right operators}.

  Similarly, for a parking function $\rho$ and a descent-intermediary position $j$ of $\rho$, let $j'$ be the largest integer such that $0\leq j' < j-1$ and $\rho(j') < \rho(j)$. Suppose that $\rho(j'+1)=\rho(j'+2)=\ldots=\rho(j'+\ell)< \rho(j'+\ell+1)$ for some integer $\ell\geq 1$. Define $L_j(\rho)\in \PF_n$ to be the parking function obtained by composing $\rho$ with the cyclic permutation $(j \ j-1 \ \ldots \ j'+\ell+1)$: \[L_{j}(\rho)\coloneq\rho\circ (j \ j-1 \ \ldots \  j'+\ell+1).\]
  For an integer $k\in [n]$ which is not a descent-intermediary position of $\rho$, let $L_k(\rho)\coloneq\rho$. For an integer $j\in [n]$, let $L_j:\PF_n \to \PF_n$ be the map given by sending $\rho \mapsto L_j(\rho)$. Again, for reasons we will shortly explain, we refer to the functions $L_1,L_2,\ldots, L_n$ as \emph{left operators}.
\end{defn}

We now explain the terminology ``left and right operators.'' 
Let $i$ be an ascent-intermediary position of $\rho$. One can visualize the parking function $R_i(\rho)$ in terms of the mountain diagrams $M_\rho$ and $M_{R_i(\rho)}$ of $\rho$ and $R_i(\rho)$. The mountain diagram of $R_i(\rho)$ is the one obtained from $M_\rho$ by moving the node $(i,\rho(i))$ of $M_\rho$ rightwards, passing under any nodes at the same height, until it hits the next down-slope of $M_\rho$ (if we hit a tie on the down-slope, the node $(i,\rho_i)$ is moved to the right-hand side of this tie). Similarly, if $j$ is a descent-intermediary position of $\rho$, $L_j(\rho)$ has the mountain diagram obtained from $M_\rho$ by moving the node $(j,\rho(j))$ leftwards, passing under any nodes at the same height, until it hits the next up-slope of $M_\rho$ (if we hit a tie on the up-slope, the node $(i,\rho_i)$ is moved to the right-hand side of this tie).

\begin{example}
  Consider the parking function $\rho=(1,2,2,4,3,2,2,3,2)$ and the ascent-intermediary position $3$. The smallest position larger than $3$ for which $\rho$ takes on a value smaller than $\rho(3)=2$ is $10$, where $\rho(10)=0$. Thus
  \[R_3(\rho) = \rho \circ (3\ 4\ 5\ 6\ 7\ 8\ 9) = (1,2,4,3,2,2,3,2,2). \]
  We draw a pictorial representation of applying $R_3$ to $\rho$ in Figure~\ref{fig:rightOperator}. 
  \begin{figure}[!h]
\centering
\includegraphics[scale = 0.35]{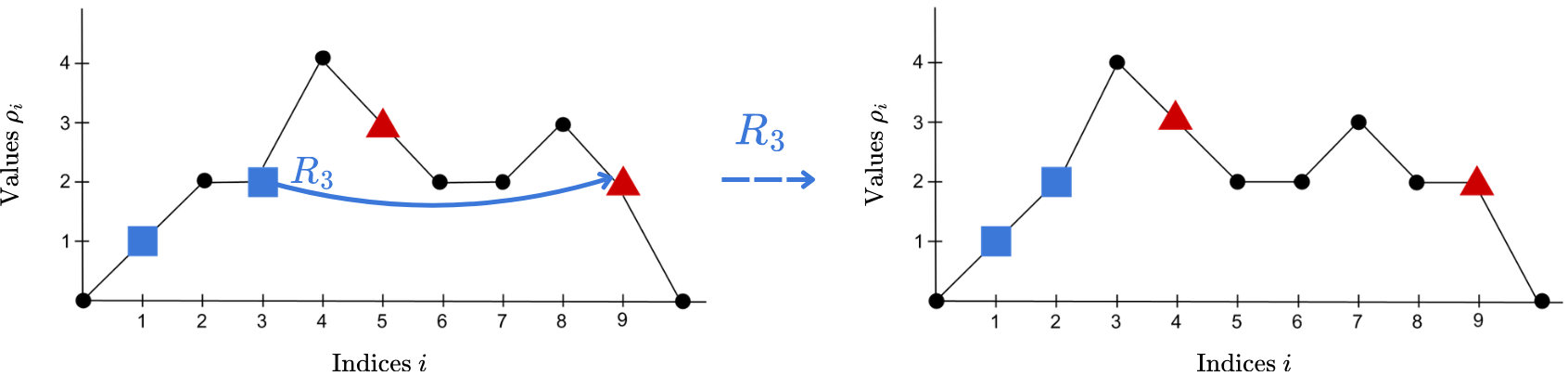}
\caption{Applying the right operator $R_3$ to the parking function $(1,2,2,4,3,2,2,3,2)$. Ascent-intermediary positions correspond to blue squares while descent-intermediary positions correspond to red triangles.}
  \label{fig:rightOperator}
\end{figure}
\end{example}

\begin{lemma}\label{lem:idempotent}
  If $i$ is an ascent-intermediary position of $\rho$ and $i'$ is as defined in Definition~\ref{defn:lr}, then $i'-1$ is a descent-intermediary position of $R_i(\rho)$, and the composition of the corresponding left and right operators is the identity:
  \[L_{i'-1}(R_i(\rho)) = \rho. \]
  Similarly, if $j$ is a descent-intermediary position of $\rho$ and $j'+\ell$ is as defined in Definition~\ref{defn:lr}, then $j'+\ell+1$ is an ascent-intermediary position of $L_j(\rho)$, and the composition of the corresponding left and right operators is the identity:
  \[R_{j'+\ell+1}(L_j(\rho)) = \rho. \]
\end{lemma}
\begin{proof}
  We prove the ascent-intermediary case, the proof of the descent-intermediary claim is identical. It is a consequence of the definitions that $i'-1$ is a descent-intermediary element of $R_i(\rho)$. Furthermore, it follows from the definitions that
  \[L_{i'-1}(R_i(\rho)) = \rho \circ (i\ i+1\ \ldots \ i'-2\ i'-1) \circ (i'-1 \ i'-2 \ \ldots \ i+1\ i). \]
  As the cycles $(i\ i+1\ \ldots \ i'-2\ i'-1)$ and $(i'-1 \ i'-2 \ \ldots \ i+1\ i)$ can be obtained from each other by reversing their orders, their composition is the identity.
\end{proof}

As the left and right operators are compositions of consecutively increasing/decreasing cycles, after some reindexing, the operations commute. We record this observation in the following lemma.

\begin{lemma}\label{lem:commute}
  Let $i<j$ be two different ascent-intermediary positions of a parking function $\rho$. Suppose that $R_i(\rho) = \rho \circ \sigma_i$ and $R_j(\rho) = \rho \circ \sigma_j$ where $\sigma_i = (i\ i+1 \ \ldots\ i'-1)$ and $\sigma_j= (j\ j+1\ \ldots \ j'-1)$ are the cyclic permutations defined in Definition~\ref{defn:lr}. The positions $\sigma^{-1}_j(i)$ and $\sigma^{-1}_i(j)$ are ascent-intermediary positions of $R_j(\rho)$ and $R_i(\rho)$, respectively. Furthermore, we have the equality
  \[R_{\sigma_j^{-1}(i)}(R_{j}(\rho)) = R_i(R_j(\rho)) = R_{\sigma_i^{-1}(j)}(R_i(\rho)). \]
  Similarly, if $i$ and $j$ are two different descent intermediary positions of $\rho$, $L_i(\rho) = \rho \circ \sigma_i$, and $L_j=\rho \circ \sigma_j$, then
  \[L_{\sigma_j^{-1}(i)}(L_{j}(\rho)) = L_j(L_i(\rho)) = L_{\sigma_i^{-1}(j)}(L_i(\rho)). \]
\end{lemma}
\begin{proof}
  We prove the ascent-intermediary case. The proof of the descent-intermediary case is identical. As $j$ is ascent-intermediary, we cannot have $i'-1 = j$. If $i'-1< j$, then $\sigma_i$ and $\sigma_j$ are disjoint cycles and the claim is immediate. If $i'-1>j$, then $\rho(j)\geq \rho(i)$ and $i'-1 > j'-1$. In this case, we have the chain of inequalities $i<j<j'-1<i'-1$. From this chain of inequalities, we see that $\sigma_j^{-1}(i)$ and $\sigma_i^{-1}(j)$ are ascent-intermediary positions of $R_j(\rho)$ and $R_i(\rho)$, respectively. Furthermore, we have
  \[R_{\sigma_j^{-1}(i)}(R_j(\rho))= R_i(R_j(\rho)) = \rho \circ(j\ j+1\ \ldots \ j'-1)\circ (i\ i+1 \ \ldots\ i'-1)\]
  and
  \[R_{\sigma_i^{-1}(j)}(R_i(\rho))= \rho \circ (i\ i+1\ \ldots\ i'-1) \circ (j-1\ j\ \ldots \ j'-2), \]
  where  $(i\ i+1\ \ldots \ i'-1) \circ (j-1\ j \ \ldots \ j'-2) = (j\ j+1\ \ldots \ j'-1)\circ (i\ i+1 \ \ldots\ i'-1)$ so the claim follows.
\end{proof}

In the proof of the next lemma, we will use the following definition.

\begin{defn}
A position $j$ of a parking function $\rho$ is \emph{tied with a descent-intermediary position} if there exists a descent-intermediary position $i$ and an integer $\ell\geq 0$ such that $j+\ell =i$ and $\rho(j)=\rho(j+1)=\ldots=\rho(j+\ell-1)=\rho(i)$. 
\end{defn}

\begin{lemma}\label{lem:uniqueNoDI}
   Given any parking function $\rho\in \PF_n$, there exists a unique parking function $L_{i_1}\circ L_{i_2}\circ \cdots \circ L_{i_k} (\rho)$ obtainable from $\rho$ by a sequence of left operations such that $L_{i_1}\circ L_{i_2}\circ \cdots \circ L_{i_k} (\rho)$ has no descent-intermediary positions.
 \end{lemma}

 \begin{proof}
Applying a left operator reduces the number of positions that are tied with a descent-intermediary position by one. If $\rho$ has no positions that are tied with a descent-intermediary position, then, in particular, $\rho$ has no descent-intermediary positions. This guarantees the existence of a parking function $L_{i_1}\circ L_{i_2}\circ \cdots \circ L_{i_k} (\rho)$ which has no descent-intermediary positions and is obtainable from $\rho$ by a sequence of left operations. That $L_{i_1}\circ L_{i_2}\circ \cdots \circ L_{i_k} (\rho)$ is unique follows from Lemma~\ref{lem:commute}.
 \end{proof}

\begin{defn}
 For a subset $S=\{s_1< s_2<\ldots < s_k\}\subseteq [n]$, define $R_S:\PF_n\to \PF_n$ to be the composition
 \[R_S\coloneq R_{s_1} \circ R_{s_2} \circ \cdots \circ R_{s_k}: \PF_n \to \PF_n. \]
 \end{defn}
 Lemma~\ref{lem:commute} ensures that if $S=\{s_1< s_2<\ldots < s_k\}$ is a subset of the ascent-intermediary elements of $\rho\in \PF_n$, then every $s_i\in S$ is an ascent-intermediary element of $R_{\{s_{i+1},\ldots, s_k\}}(\rho)= R_{s_{i+1}}\circ \cdots \circ R_{s_k}(\rho)$. However, as the following example shows, the converse is not true; it is possible to have a parking function $\rho$ and indices $i<j$ such that $j$ is an ascent-intermediary position of $\rho$ and $i$ is an ascent-intermediary position of $R_j(\rho)$, but $i$ is not an ascent-intermediary position of $\rho$.
 \begin{example}\label{ex:airevealed}
   Consider the parking function $\rho=(1,2,2,4,3,2,2,3,2)$ and the positions $2$ and $3$. Although $3$ is an ascent-intermediary position of $\rho$ and $2$ is ascent-intermediary position of $R_3(\rho)$, $2$ is not an ascent intermediary position of $\rho$. See Figure~\ref{fig:rightOperator}.  This example also shows that the hypotheses of Lemma~\ref{lem:commute} are necessary. As
   \[R_2(R_3(\rho))= (1,4,3,2,2,3,2,2,2) \neq (1,2,4,3,2,2,3,2,2) = R_2(R_3(\rho)),\]
   the compositions of the right operators $R_2$ and $R_3$ do not commute.
 \end{example}

 From here on out, we will first fix a parking function $\rho$ and consider the parking functions $R_S(\rho)$ as $S$ ranges over all subsets of the ascent-intermediary positions of $\rho$. 
 \begin{defn}
   Let $\rho\in \PF_n$ and let $A\subseteq [n]$ denote its set of ascent-intermediary positions. The set of parking functions \emph{reachable} from $\rho$ is
   \[\mathcal{R}(\rho) = \{R_S(\rho): S\subseteq A\} \]
 \end{defn}
 
 \begin{example}\label{ex:reachableSets}
   Let $\rho$ be the length $9$ parking function $(1,2,2,4,3,2,2,3,2)$. The ascent-intermediary positions of $\rho$ are $1$ and $3$. The set $\mathcal{R}(\rho)$ is equal to $\{\rho,R_1(\rho),R_3(\rho), R_{\{1,3\}}(\rho)\}$ and is illustrated in Figure~\ref{fig:reachableFigure}.
\begin{figure}[!h]
\centering
\includegraphics[scale = 0.4]{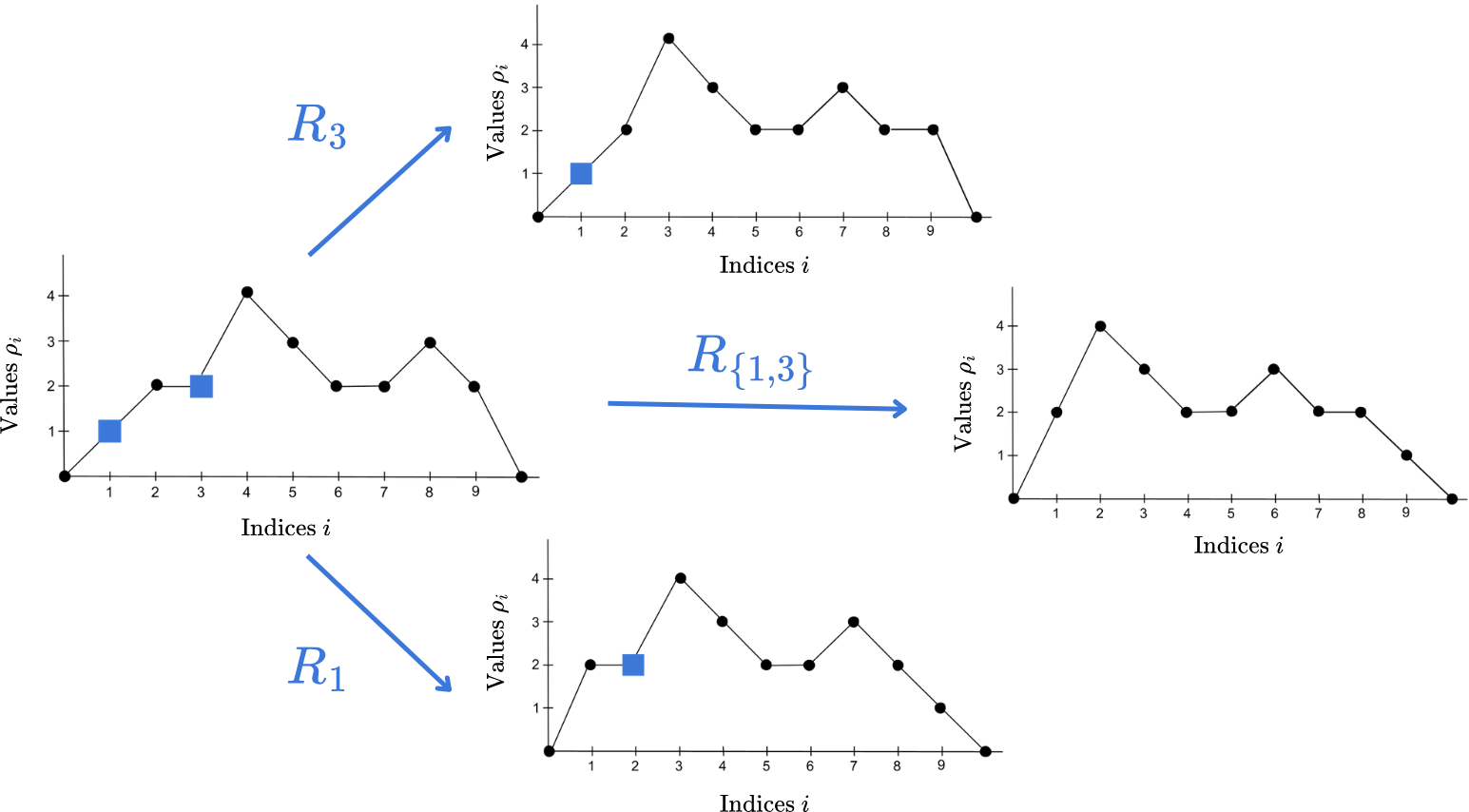}
  \caption{The set $\mathcal{R}(\rho)$ where $\rho$ is the parking function $(1,2,2,4,3,2,2,3,2)$. The blue squares correspond to ascent intermediaries of $\rho$.}
  \label{fig:reachableFigure}
\end{figure}
\end{example}

\begin{lemma} \label{lem:allDiff}
  Let $\rho\in \PF_n$ and let $A$ denote the set of ascent-intermediary positions of $\rho$. The number of reachable parking functions from $\rho$ is equal to
  \[\# \mathcal{R}(\rho) = 2^{\# A}. \]
\end{lemma}
\begin{proof}
  We need to show that if $S$ and $S'$ are two different subsets of $A$, then $R_S(\rho) \neq R_{S'}(\rho)$. Applying a right operator increases the number of positions tied with a descent-intermediary position, thus in order to have $R_{S}(\rho) = R_{S'}(\rho)$, we must have $\# S= \# S'$. Suppose that $S = \{s_1< s_2< \ldots < s_k\}$ and $S'= \{s_1'<s_2'<\ldots <s_k'\}$ where $s_i< s_i'$ and $i$ is the first index where they differ. For $j<i$, let $\sigma_j$ be the cyclic permutation such that $R_{s_j}(\rho) = \rho \circ \sigma_j$. We have that
  \[R_S(\rho)(\sigma_1^{-1}\sigma_2^{-1}\cdots \sigma_{i-1}^{-1}(s_i)) > R_{S'}(\rho)(\sigma_1^{-1}\sigma_2^{-1}\cdots \sigma_{i-1}^{-1}(s_i)) \]
  and hence $R_S(\rho)\neq R_{S'}(\rho)$.
\end{proof}

The next two propositions are key to our proof of Theorem~\ref{thrm:main2}.

\begin{proposition}\label{prop:countDes}
 If $\rho\in \NDWD_n$, then $\rho$ has no descent-intermediary positions, and there is an equality of polynomials
   \[\sum_{\rho' \in \mathcal{R}(\rho)\cap \NT_n}x^{\des(\rho')}= x^{\wdes(\rho)}(x+1)^{n-2\wdes(\rho)-1}\,. \]
 \end{proposition}
 \begin{proof}
   That $\rho$ has no descent-intermediary positions follows directly from the definitions. We now focus on proving the equality of polynomials. Let $A$ be the set of ascent-intermediary positions of $\rho$ and let $T\subseteq[n]$ be the set of ends of ties of $\rho$, i.e., $T= \{i\in [n] : \rho(i-1)=\rho(i)\}$.
As $\rho$ has no consecutive weak descents, $T\subseteq A$.
   
Let $i$ be an ascent-intermediary position of $\rho$ and $\sigma_i= (i\ i+1\ \ldots \ i'-1)$ be the cyclic permutation defined in Definition~\ref{defn:lr} such that $R_i(\rho) = \rho \circ \sigma_i$. We claim that $R_i(\rho)$ has one more strict descent than $\rho$. To check this, we need to verify that the position $i'-1$ is not the end of a tie of $R_i(\rho)$. Suppose for contradiction that $R_i(\rho)(i'-2)=R_i(\rho)(i'-1)$. This is equivalent to the assumption that $\rho(i'-1)=\rho(i)$. By the definition of $i'$, $\rho(i'-2)\geq \rho(i)$ and $\rho(i')<\rho(i)$. Thus we get the chain of inequalities $\rho(i'-2)\geq \rho(i'-1) > \rho(i')$, which contradicts the fact that $\rho$ has no consecutive weak descents.

A similar argument lets us see that for any $S\subseteq A$, the number of strict descents in $R_S(\rho)$ is equal to $\des(\rho)+ \#S$. Furthermore, the intersection $\mathcal{R}(\rho)\cap \NT_n$ is equal to
\[\mathcal{R}(\rho) \cap \NT_n = \{ R_{S \cup T}(\rho): S \subseteq A\setminus T \}. \]
   
    As $\tie(\rho)+\des(\rho) = \wdes(\rho)$, it follows from Lemma~\ref{lem:allDiff} that
   \[\sum_{\rho' \in \mathcal{R}(\rho)\cap \NT_n}x^{\des(\rho')} = x^{\wdes(\rho)}\sum_{S\subseteq A\setminus T} x^{\#S}= x^{\wdes(\rho)}(x+1)^{\#(A\setminus T)}. \]
   All that remains to show is that $\#(A\setminus T)= n-2\wdes(\rho)-1$.
   As $\rho$ has no consecutive weak descents, every descent of $\rho$ is followed by an ascent and any ascent which does not immediately precede a descent is an ascent-intermediary position of $\rho$. This ensures that the number of ascent-intermediary positions of $\rho$ is equal to $\asc(\rho)-\des(\rho)$. Therefore
   \begin{align*}
     (n-1)-2\wdes(\rho) &= \asc(\rho)-\wdes(\rho)\\
                     &= \asc(\rho)-\des(\rho) -\tie(\rho)\\
                     &= \#A-\tie(\rho)\\
                     &= \#(A\setminus T).
   \end{align*}
   \end{proof}

\begin{proposition}\label{prop:tpfContained}
   Every tieless parking function is reachable from a parking function $\rho\in \NDWD_n$ with no consecutive weak descents. Furthermore, if $\rho$ and $\rho'$ are two different parking functions in $\NDWD_n$, then
   \[\mathcal{R}(\rho) \cap \mathcal{R}(\rho') = \emptyset. \]
 \end{proposition}
 \begin{proof}
   We begin with the first claim. Let $\rho''$ be a tieless parking function. As $\rho''$ has no ties, if $\{d_1<d_2<\ldots <d_k\}$ is its set of descent-intermediary positions, then the parking function $\rho\coloneq L_{d_k}\circ L_{d_{k-1}} \circ \cdots \circ L_{d_1}(\rho'')$ obtained by composing all of the corresponding left operators has no descent-intermediary entries. By Lemma~\ref{lem:idempotent}, this implies that $\rho''$ is reachable from $\rho$. We claim that $\rho$ has no consecutive weak descents. As $\rho''$ is tieless and reachable from $\rho$, we must have that whenever $\rho(i)=\rho(i+1)$, $\rho(i-1) < \rho(i)$ and $\rho(i+1) < \rho(i+2)$.  In particular, if $\rho(i-1)\geq \rho(i) \geq \rho(i+1)$, then each of the inequalities is strict. However, if $\rho(i-1)>\rho(i)>\rho(i+1)$, then $i$ is a descent-intermediary position of $\rho$, which is a contradiction.

   To see the second claim, note that from any parking function $\tilde \rho$ in $\mathcal{R}(\rho)$, we can obtain $\rho$ by applying a sequence of left operators (simply reverse the order in which we applied the right operators to get from $\rho$ to $\tilde \rho$). If $\tilde \rho \in \mathcal{R}(\rho)\cap \mathcal{R}(\rho')$, the previous observation would contradict Lemma~\ref{lem:uniqueNoDI}.
\end{proof}

We are now in a position to prove Theorem~\ref{thrm:main2}.
\begin{proof}[Proof of Theorem~\ref{thrm:main2}]
By Proposition~\ref{prop:tpfContained}, every tieless parking function is reachable from a unique parking function with no consecutive weak descents. The result now follows from Proposition~\ref{prop:countDes}.
\end{proof}
We illustrate Theorem~\ref{thrm:main2} in the case where $n=3$.
\begin{example}\label{ex:mainTheorem}
    There are $7$ length $3$ parking functions that have no consecutive weak descents. In Figure~\ref{fig:gammaTable}, we list these parking functions, the size of their reachable sets, and their contributions to the $\gamma$-expansion of $H_{\NC_4}(x)$. As one can calculate by summing the sizes of reachable sets, not every parking function of length $3$ is reachable from a parking function in $\NDWD_3$. For example, the parking function $(3,1,1)$ is not reachable from a parking function with no consecutive weak descents.

\begin{figure}[h]
\centering
\begin{tabular}{c|c|c|c}
$\rho\in \NDWD_3$ & $\# \mathcal{R}(\rho)$ & $\# \mathcal{R}(\rho) \cap \NT_3$ & Contribution to the $\gamma$-expansion of $H_{\NC_4}(x)$ \\ 
\hline
$(1,1,2)$ & $2$ & $1$ & $x$ \\
$(1,1,3)$  & $2$  & $1$ & $x$ \\
$(1,2,3)$  & $4$  & $4$ & $(x+1)^2$ \\    
$(2,1,2)$  & $1$  & $1$ & $x$ \\
$(2,1,3)$  & $1$  & $1$ & $x$ \\
$(3,1,2)$  & $1$  & $1$ & $x$ \\
\end{tabular}
\caption{A table listing the parking functions $\rho\in \NDWD_3$, the sizes of $\mathcal{R}(\rho)$, the sizes of $\mathcal{R}(\rho) \cap \NT_3$,  and the contribution of $\mathcal{R}(\rho)\cap \NT_3$ to the $\gamma$-expansion of $H_{\NC_{4}}(x)$.}\label{fig:gammaTable}
\end{figure}

\end{example}

\section{The parking function labeling of $\NC_{n+1}$ and a proof of Theorem~\ref{thrm:main1}}\label{sec:pf}

In this section, we review Stanley's R-labeling of $\NC_{n+1}$ indexed by parking functions. Combining this labeling with Theorem~\ref{thrm:stump}, we obtain a $\gamma$-expansion of $H_{\NC_{n+1}}$ in terms of parking functions. We then use this $\gamma$-expansion, along with Theorem~\ref{thrm:main2}, to prove Theorem~\ref{thrm:main1}.

In \cite[Section 3]{Stanley}, Stanley constructs an R-labeling of the noncrossing partition lattice indexed by parking functions. He does this in two steps. First, he constructs a labeling $\Lambda^{\PF}$ of $\NC_{n+1}$ such that the labels $\Lambda^{\PF}(\mathbf{m})$ of the maximal chains of $\NC_{n+1}$ are length $n$ parking functions.  Next, he constructs a different labeling $\Lambda^{\R}$ of $\NC_{n+1}$ which is closely related to $\Lambda^{\PF}$ and is an R-labeling of $\NC_{n+1}$.  We recall both $\Lambda^{\PF}$ and $\Lambda^{\R}$ below. 

\begin{defn}
Let $(\pi,\sigma) \in \mathcal{E}(\NC_{n+1})$. In this case, $\sigma$ is obtained from $\pi$ by merging two blocks $B_1$ and $B_2$ of $\pi$. Suppose that $\min(B_1)< \min(B_2)$. Define the labeling $\Lambda^{\PF}:\mathcal{E}(\NC_{n+1})\to \mathbb{Z}$ by
\[\Lambda^{\PF}(\pi,\sigma) = \max\{i \in B_1: i < \min(B_2)\}.\]
Define the labeling $\Lambda^{\R}:\mathcal{E}(\NC_{n+1})\to \mathbb{Z}$ by 
\[\Lambda^{\R}(\pi,\sigma) = \vert \pi \vert - \Lambda^{\PF}(\pi,\sigma) \]
where $\vert \pi \vert$ is the number of blocks of $\pi$.
\end{defn}

\begin{thrm}{\cite[Theorem 3.1]{Stanley}}
\label{thrm:stanley}
    The labelings $\Lambda^{\PF}(\mathbf{m})$ of the maximal chains of $\NC_{n+1}$ consist of the parking functions of length $n$, each occurring once. The labeling $\Lambda^{\R}$ is an R-labeling.
\end{thrm}

We illustrate the labelings $\Lambda^{\PF}$ and $\Lambda^{\R}$, as well as \Cref{thrm:stanley}, in the below example.

\begin{example}
    In Figure~\ref{fig:Rlabels}, we illustrate the labelings $\Lambda^{\PF}$ and $\Lambda^{\R}$ for $\NC_3$. The three maximal chains of $\NC_3$ are labeled by $\Lambda^{\PF}$ with the three parking functions in $\PF_2$. The maximal chain 
    $\mathbf{m} = (1\vert 2 \vert 3 < 1 \vert 23 < 123)$
    has label $\Lambda^{\PF}(\mathbf{m})=(2,1)$, the unique strictly decreasing parking function in $\PF_2$. In turn, Lemma~\ref{lem:descents} tells us that  $\Lambda^{\R}(\mathbf{m})=(1,1)$ is the unique weakly increasing maximal chain label of $\NC_3$, which is indeed the case. 

    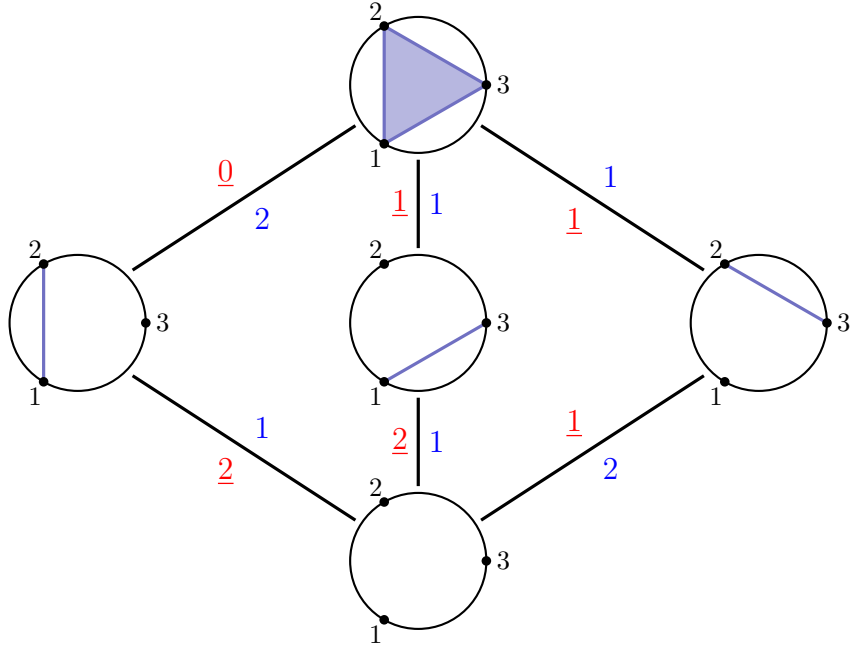
\begin{figure}[h]
        \centering
        \begin{tikzpicture} [scale=0.9]
        \draw[thick] circle (1cm);
            \foreach \x in {1,2,3} {
            \draw node at (-\x*1/3*360:1.25cm) {$\x$};
            \fill (-\x*1/3*360:1cm) circle (2pt);
        };
        \draw[very thick] (147:1.1cm) -- (147: 5cm) node[midway,above right] {\color{blue}\Large 1} node[midway,below left] {\color{red}\Large \underline 2};
        \draw[very thick] (0,1.1cm) -- (0,2.4cm) node[midway,right] {\color{blue}\Large 1} node[midway,left] {\color{red}\Large \underline 2};
        \draw[very thick] (33:1.1cm) -- (33: 5cm) node[midway,below right] {\color{blue}\Large 2} node[midway,above left] {\color{red}\Large \underline 1};

        \filldraw[very thick, draw opacity = 0.8, fill opacity =0.4, color=blue!40!gray] ($(-5,3.5cm)+(1/3*360:1cm)$)--($(-5,3.5cm)+(2/3*360:1cm)$);
         \draw[thick] (-5cm,3.5cm) circle (1cm);
            \foreach \x in {1,2,3} {
            \draw node at ($(-5cm,3.5cm)+(-\x*1/3*360:1.25cm)$) {$\x$};
            \fill ($(-5,3.5cm)+(-\x*1/3*360:1cm)$) circle (2pt);
        };

        \filldraw[very thick, draw opacity = 0.8, fill opacity =0.4, color=blue!40!gray] ($(0,3.5cm)+(2/3*360:1cm)$)--($(0,3.5cm)+(3/3*360:1cm)$);
        \draw[thick] (0,3.5cm) circle (1cm);
            \foreach \x in {1,2,3} {
            \draw node at ($(0,3.5cm)+(-\x*1/3*360:1.25cm)$) {$\x$};
            \fill ($(0,3.5cm)+(-\x*1/3*360:1cm)$) circle (2pt);
        };
        
        \filldraw[very thick, draw opacity = 0.8, fill opacity =0.4, color=blue!40!gray] ($(5,3.5cm)+(1/3*360:1cm)$)--($(5,3.5cm)+(3/3*360:1cm)$);
        \draw[thick] (5cm,3.5cm) circle (1cm);
            \foreach \x in {1,2,3} {
            \draw node at ($(5cm,3.5cm)+(-\x*1/3*360:1.25cm)$) {$\x$};
            \fill ($(5,3.5cm)+(-\x*1/3*360:1cm)$) circle (2pt);
        };
        
        \filldraw[very thick, draw opacity = 0.8, fill opacity =0.4, color=blue!40!gray] ($(0,7cm)+(2/3*360:1cm)$)--($(0,7cm)+(3/3*360:1cm)$)-- ($(0,7cm)+(1/3*360:1cm)$) -- cycle;
        \draw[thick] (0,7cm) circle (1cm);
            \foreach \x in {1,2,3} {
            \draw node at ($(0,7cm)+(-\x*1/3*360:1.25cm)$) {$\x$};
            \fill ($(0,7cm)+(-\x*1/3*360:1cm)$) circle (2pt);
        };
        
        \draw[very thick] ($(0,7cm)+(33:-1.1cm)$) -- ($(0,7cm)+(33: -5cm)$) node[midway,below right] {\color{blue}\Large 2} node[midway,above left] {\color{red}\Large \underline 0};
        \draw[very thick] ($(0,7cm)+(0,-1.1cm)$) -- ($(0,7cm)+(0,-2.4cm)$) node[midway,right] {\color{blue}\Large 1} node[midway,left] {\color{red}\Large \underline 1};
        \draw[very thick] ($(0,7cm)+ (147:-1.1cm)$) -- ($(0,7cm)+(147: -5cm)$) node[midway,above right] {\color{blue}\Large 1} node[midway,below left] {\color{red}\Large \underline 1};
        \end{tikzpicture}
        
        \caption{The Hasse diagram of $NC_3$ with the labeling $\Lambda^{\PF}$ written in blue on the right and the labeling $\Lambda^{\R}$ written in red, underlined and on the left.}
        \label{fig:Rlabels}
    \end{figure}
\end{example}

As $\Lambda^{\PF}$ and $\Lambda^{\R}$ are so closely related, we can read off the descents of $\Lambda^{\R}$ from the data of $\Lambda^{\PF}$.

\begin{lemma} \label{lem:descents}
    Let $\mathbf{m}$ be a maximal chain of $\NC_{n+1}$. The labeling $\Lambda^{\R}(\mathbf{m})$ has a descent $\Lambda^{\R}(\mathbf{m})_i > \Lambda^{\R}(\mathbf{m})_{i+1} $ at index $i$ if and only if the labeling $\Lambda^{\PF}(\mathbf{m})$ has a weak ascent $\Lambda^{\PF}(\mathbf{m})_{i} \leq \Lambda^{\PF}(\mathbf{m})_{i+1}$ at index $i$.
\end{lemma}
\begin{proof}
  Let $\mathbf{m} = \pi_0< \pi_1 <\cdots <\pi_n$. The quantity $\Lambda^{\R}(\pi_{i-1},\pi_i) - \Lambda^{\R}(\pi_i,\pi_{i+1})$ is equal to
  \begin{align*}
    \Lambda^{\R}(\pi_{i-1},\pi_i) - \Lambda^{\R}(\pi_i,\pi_{i+1}) &= |\pi_{i-1}| - |\pi_{i}| - (\Lambda^{\PF}(\pi_{i-1},\pi_i) - \Lambda^{\PF}(\pi_i,\pi_{i+1}))\\
                                        &= \Lambda^{\PF}(\pi_i,\pi_{i+1})- \Lambda^{\PF}(\pi_{i-1},\pi_i) + 1.
  \end{align*}
  This expression is positive if and only if $\Lambda^{\PF}(\pi_{i-1},\pi_i) \leq \Lambda^{\PF}(\pi_i,\pi_{i+1})$.
\end{proof}

Combining \Cref{thrm:stump} and \Cref{lem:descents}, we obtain an expression for the Chow polynomial of $\NC_{n+1}$ in terms of parking function statistics.

\begin{proposition}\label{prop:PFChow}
The Chow polynomial $H_{NC_{n+1}}(x)$ has the expansion
\[H_{NC_{n+1}}(x) = \sum_{\rho\in \NDWD_n} x^{\wdes(\rho)} (x + 1)^{n - 1 - 2\wdes(\rho)}.\]
\end{proposition}
\begin{proof}
Applying Theorem~\ref{thrm:stump} and Lemma~\ref{lem:descents} to $\Lambda^{\R}$, we see that 
    \begin{equation}\label{eq:ascents}H_{NC_{n+1}}(x) = \sum_{\rho\in \mathcal{P}} x^{\wasc(\rho)} (x + 1)^{n - 1 - 2\wasc(\rho)}\end{equation}
where $\mathcal{P}$ is the subset of $\PF_n$ consisting of all parking functions $\rho = (\rho_1, \rho_2, ..., \rho_n)$ such that the following conditions hold:
\begin{itemize}
    \item $\rho_1 > \rho_2$, and
    \item There is no index $i\in \{2,3,\ldots,n-1\}$ such that $\rho_{i-1} \leq \rho_i \leq \rho_{i+1}$.
\end{itemize}
Let $rev:\PF_n\to \PF_n$ be the function which reads a parking function in reverse. That is,  $rev(\rho_1,\rho_2,\ldots, \rho_n)=(\rho_n, \rho_{n-1}, \ldots, \rho_1)$. The map $rev$ restricts to a bijection between $\mathcal{P}$ and $\NDWD_n$, and sends weak ascents to weak descents. Applying the map $rev$ to \Cref{eq:ascents} yields the claim.
\end{proof}

We can now proceed to prove Theorem~\ref{thrm:main1} and Corollary~\ref{thrm:linearCoeff}.

\begin{proof}[Proof of Theorem~\ref{thrm:main1}]
    Compare Proposition~\ref{prop:PFChow} with Theorem~\ref{thrm:main2}.
\end{proof}

\begin{proof}[Proof of Corollary~\ref{thrm:linearCoeff}]
    The linear coefficient of the Chow polynomial $H_P(x)$ of a poset $P$ is equal to the number of elements of $P$ with rank at least two.
    As $\NC_{n+1}$ has $\binom{n+1}{2}$ atoms and size equal to $C_{n+1}$, 
    there are $C_{n+1}-\binom{n+1}{2}-1$ such elements. The claim now follows from \Cref{thrm:main1}.
\end{proof}

\printbibliography
\end{document}